\documentclass[12pt, reqno, twoside, letterpaper]{amsart}

\usepackage{paperstyle}
\usepackage{graphicx}
\usepackage{float}

\usepackage{mathrsfs}

\def \sfS{\mathsf S}

\usepackage{mathtools}
\usepackage{todonotes}
\usepackage[norefs,nocites]{refcheck}

\newcommand{\be}{\begin{equation}}
\newcommand{\ee}{\end{equation}}
\newcommand{\dalign}[1]{\[\begin{aligned} #1\end{aligned}\]}
\newcommand{\er}{\mathrm{e}}
\newcommand{\llsym}[1]{\,\mathop{\ll}\limits_{#1}\,}

\newcommand{\reduceit}[2]{\llbracket #1 \rrbracket_{#2}}

\newcommand{\lcm}{\operatorname{lcm}}
\newcommand{\GL}{\operatorname{GL}}

\newcommand{\linkMR}[1]{\href{https://mathscinet.ams.org/mathscinet/relay-station?mr=#1}{#1}}

\title[Multiple sums with the M{\"o}bius function]
{Multiple sums with the M{\"o}bius function}
 
\author[W. D. Banks]{William D. Banks}
\address{Department of Mathematics, 
University of Missouri, 
Columbia MO 65211, USA.}
\email{bankswd@missouri.edu}

\author[I. E. Shparlinski] {Igor E. Shparlinski}
\address{School of Mathematics and Statistics,
University of New South Wales,
Sydney NSW 2052, Australia}
\email{igor.shparlinski@unsw.edu.au}
 
\date{\today}

\begin{document}

\begin{abstract}
We establish nontrivial bounds for bilinear sums
involving the M\"obius function 
evaluated over solutions to a broad class of equations.
Several of our results may be regarded as M\"obius-function
analogues of the ternary Goldbach problem. 
By contrast, the binary versions of our results remain out of 
reach, much like the binary Goldbach problem. Nevertheless, we make 
partial progress in this direction by restricting the range of the 
third variable as far as possible.
\end{abstract}  



\date{}

\maketitle


\tableofcontents

\newpage{\large\section{Introduction}}

\subsection{General conventions}
Throughout the paper, 
the letters $p$ and $q$ 
(with or without subscripts) are used to denote
prime numbers, and all other letters denote natural
numbers unless specified otherwise.

We recall that the notations $U = O(V)$, $U \ll V$, and $ V\gg U$  
are all understood to mean that  $|U|\leqslant c V$ for some 
positive constant $c$, which may depend on the parameters indicated 
below the respective symbols. For example, in~\eqref{eq:notation 
example} (see below), the implied constant in the error term 
depends only on the real parameters $C$ and $\Delta$.

Throughout, we assume that the arguments of all logarithmic
expressions are large enough to ensure the formulas are 
well-defined.

\subsection{Motivation}
The \emph{ternary Goldbach problem} (or weak Goldbach conjecture)
asserts that every odd integer greater than five can be
expressed as the sum of three primes. The problem amounts
to establishing the positivity of the sum
\[
S_\Lambda(N)\defeq\sum_{n_1+n_2+n_3=N}
\Lambda(n_1)\Lambda(n_2)\Lambda(n_3)
\]
for every odd integer $N\ge 7$, where $\Lambda$ is the
von Mangoldt function, and the sum runs
over ordered triples $(n_1,n_2,n_3)$ of natural numbers.
In 1923, Hardy and Littlewood showed that this is true
for all sufficiently large odd $N$ under the
\emph{Generalized Riemann Hypothesis} (GRH).  
In 1937, Vinogradov eliminated the dependency on GRH, proving
unconditionally that
the weak Goldbach conjecture holds for all sufficiently
large odd numbers. In 2013, the problem was completely resolved
by Helfgott~\cite{Helf}.

Sums similar to $S_\Lambda(N)$ have been studied by various authors.
For example, Fouvry and Ganguly~\cite[eqn.~(10)]{FouvGang} have
obtained the bound
\be\label{eq:Rama}
S_\tau(N)\defeq
\sum_{p_1+p_2+p_3=N}\tau(p_1)\ll N^{15/2}\er^{-c\sqrt{\log N}}
\ee
where $\tau$ is the Ramanujan $\tau$-function,
and the sum runs over ordered triples $(p_1,p_2,p_3)$ of primes; their method
also yields the same bound for sums with $\tau(p_1)\tau(p_2)\tau(p_3)$.
More recently, Gafni and Robles~\cite{GafniRobles} have investigated
additive analogues of $S_\Lambda(N)$; they have
given (cf.~\cite[Thm.~1.2]{GafniRobles}) estimates of the form
\be
\begin{split}
S_\Omega(N)& \defeq\sum_{n_1+n_2+n_3=N}\Omega(n_1)\Omega(n_2)\Omega(n_3)\\
\label{eq:notation example}
& =\mathfrak S(N,C,\Delta)\frac{N^2}{2}
+O_{C,\Delta}\(\frac{N^2(\log\log N)^3}{(\log N)^C}\),
\end{split}
\ee
where for each $C>0$, the singular series $\mathfrak S(N,C,\Delta)$
appearing in the main term is defined explicitly, albeit in a somewhat technical manner, in 
terms of $C$ and another flexible parameter
$\Delta\in(0,\frac12)$.

We remark that binary analogues of these results
(i.e., those involving only two variables) remain out 
of reach, the situation being similar to the binary Goldbach
problem. Nonetheless, there are several works that make partial 
progress by imposing restrictions on the third variable;
see, for example,~\cite{Cai}. In~\S\ref{sec:small var},
we show that our results allow much stronger restrictions 
in the case of the M\"obius function
$\mu$ than in the case of the von Mangoldt function $\Lambda$.
More specifically, a significant part of 
our motivation comes from the problem
of finding nontrivial upper bounds for the sums
\be
\label{eq:Short Var}
S_\mu(H,N)\defeq
\ssum{
n_1+n_2+n_3=N\\n_3\le H}\mu(n_1n_2n_3) 
\ee
with $H$ as small as possible. We succeed in establishing such
bounds with $H$ logarithmically small relative to $N$; 
see~\S\ref{sec:small var}  for details.


\subsection{Set-up}
In the present paper, we study analogues of $S_\Lambda(N)$
with the M{\"o}bius function $\mu$. Our results are broadly
applicable and are formulated in terms of a general
collection of data
\be\label{eq:Data}
{\mathbf D}\defeq(\cA,\cB, \cN,\wp,{\tt u},{\tt v},f,g)
\ee 
consisting of
\begin{itemize}
\item Finite intervals $\cA,\cB,\cN\subseteq \N$;  
\item A bracket polynomial $\wp:\Z\to\Z$ of complexity $\delta$ (see below for a definition of complexity);
\item Complex weights ${\tt u}=\{{\tt u}_n\}$
and ${\tt v}=\{{\tt v}_n\}$ with $\ell^\infty$ norms
satisfying $\|{\tt u}\|_\infty\le 1$ and $\|{\tt v}\|_\infty\le 1$;
\item Injective functions
\[
f:\cA\cap{\rm Supp}({\tt u})\to\Z,\qquad
g:\cB\cap{\rm Supp}({\tt v})\to\Z.
\]
\end{itemize}

We recall that a \emph{bracket polynomial} $p:\Z\to\R$
is an object  formed
from the scalar field $\R$ and an indeterminate $n$ using
finitely many instances of the standard
arithmetic operations $+$, $\times$ along with the integer part
operation $\fl{\cdot}$ and the fractional part operation $\{\cdot\}$;
see~\cite{BergLieb,GreenTao}. An ``ordinary''
polynomial $P(n)$ (i.e., $P\in\R[X]$)
is a bracket polynomial. More exotic examples of bracket polynomials
include $\{n^3\sqrt{7}+\fl{\pi n}\}$ and
$n\sqrt{3}\fl{n\sqrt{5}}\{\er\,n^{11}\}$. As in~\cite{GreenTao}, 
we define the \emph{complexity} $\delta$ of a bracket
polynomial $\wp$ to be the least number of operations
$+,\times,\fl{\cdot},\{\cdot\}$ required to write down $\wp$.

For a given data set ${\mathbf D}$, our primary objects of study
are the sums
\be\label{eq:SmuDM-defn}
\sfS_\mu({\mathbf D};M)\defeq
\ssum{n_1\in \cA,\,n_2\in \cB,\,n_3\in\cN\\f(n_1)+g(n_2)+\wp(n_3)=M}
{\tt u}_{n_1}{\tt v}_{n_2}\,\mu(n_1n_2n_3)\qquad(M\in\Z).
\ee
Observe that the trivial bound
\be\label{eq:triv-bd} 
\bigl|\sfS_\mu({\mathbf D};M)\bigr|
\le\min\{|\cA|,|\cB|\}\,|\cN|
\ee
follows at once from the injectivity condition on $f$ and $g$.

{\large\section{Main Results}}

\subsection{Results over long intervals}
\label{sec;long int} 

The results of this section can be applied to intervals whose
left endpoint lies near the origin.

\begin{thm}\label{thm:main1-long}
Let $C>0$, $N\ge 10$. Let ${\mathbf D}$ be
a collection  of data~\eqref{eq:Data}
such that
\[
\cA\defeq[1,A],\qquad\cB\defeq[1,B],\qquad \cN\defeq[1,N],
\]
where $\log N\ge(\log\max\{A,B\})^\eps$ with some
fixed $\eps>0$. Then,
the sums given by~\eqref{eq:SmuDM-defn} satisfy the uniform bound
\[
 \sfS_\mu({\mathbf D};M)
\llsym{C,\eps,\delta}(A+B)N(\log N)^{-C}.
\]
\end{thm}

We remark that the bound of Theorem~\ref{thm:main1-long} matches
the strength of the best known bounds for the error term in the 
asymptotic formula for $S_\Lambda(N)$; see, for instance,
Vaughan~\cite[Thm.~3.4]{Vau}.

Theorem~\ref{thm:main1-long} has numerous applications,
producing bounds that are new in many cases.
In~\S\ref{sec:applications}, we use Theorem~\ref{thm:main1-long} to
recover the bounds  
\begin{alignat*}{3}
\ssum{
p+q\le x}&\mu(p+q)
\llsym{C}x^2(\log x)^{-C}&&\qquad(p,q~\text{prime}),\\
\ssum{
k+\ell\le x}&\mu^\nu(k\ell)\,\mu(k+\ell)
\llsym{C}x^2(\log x)^{-C}&&\qquad(k,\ell\in\N,~\nu=1\text{~or~}2), 
\end{alignat*}
which are implicit in the results of~\cite{MTW}.
The following bounds, also established in~\S\ref{sec:applications},
have not appeared previously in the literature:
\begin{align*}
\ssum{p_k,p_\ell\le x\\\gcd(k\,\ell,\,p_k+p_\ell)=1}
\mu^2(k\ell)\,\mu(p_k+p_\ell) &
\llsym{C} x^2(\log x)^{-C} \qquad(p_j=\text{$j$-th prime}),\\
\ssum{a,b,c\le m\\a+b\equiv c^2\bmod m}\mu(abc) &
\llsym{C}m^2(\log m)^{-C}\qquad(a,b,c,m\in\N),
\end{align*}
As yet another example, notice that Theorem~\ref{thm:main1-long}
implies the bound
\be\label{eq:SmuNdefnbd}
S_\mu(N)\defeq\ssum{
n_1+n_2+n_3=N}\mu(n_1n_2n_3)
\llsym{C} N^2(\log N)^{-C}.
\ee
Our interest in the sum $S_\mu(N)$
served as the starting point for the present work;
the bound~\eqref{eq:SmuNdefnbd} appears to be new.
Note that $S_\mu(N)$ is equal to $S_\mu(N,N)$ in the notation
of~\eqref{eq:Short Var}.

We remark that, although the sums $S_\tau(N)$ in~\eqref{eq:Rama}
and $S_\mu(N)$ in~\eqref{eq:SmuNdefnbd}
are fundamentally different in nature, the primary factor
influencing the disparity in strength between the bounds~\eqref{eq:Rama} and~\eqref{eq:SmuNdefnbd} is the potential existence of 
\emph{Siegel zeros}. By the work of Hoffstein and
Ramakrishnan~\cite{HoffRama}, Siegel zeros cannot exist
for $L$-functions attached to cusp forms on $\GL(2)$, however they
cannot be ruled out for Dirichlet \text{$L$-functions}. 
Accordingly, if one assumes that the zeros of Dirichlet
$L$-functions are sufficiently well-behaved, then the bound~\eqref{eq:SmuNdefnbd} can be strengthened.
The following result is proved in~S\ref{sec:zeros}.

\begin{thm}\label{thm:noSiegel}
Suppose the conditions of Theorem~\ref{thm:main1-long} hold,
and assume further that $\wp$ is a linear polynomial
in $\Z[X]$. Then
\begin{itemize}
\item[$(i)$] If Siegel zeros do not exist for Dirichlet
$L$-functions, then the sums given by~\eqref{eq:SmuDM-defn} satisfy the uniform bound
\be\label{eq:uni-one}
\sfS_\mu({\mathbf D};M)
\ll(A+B)N\er^{-c\sqrt{\log N}}
\qquad(M\in\Z)
\ee 
for some effectively computable constant $c>0$.

\item[$(ii)$] If every Dirichlet $L$-function $L(s,\chi)$
is nonzero in the half-plane $\sigma\ge\sigma_\bullet$ for some
$\sigma_\bullet<1$, then
\be\label{eq:uni-two}
 |\sfS_\mu({\mathbf D};M)|
\le (A+B)N^{c(\sigma_\bullet)+o(1)} \qquad(N\to\infty),
\ee
where
\[
c(\sigma_\bullet)\defeq\begin{cases}
(4+6\sigma_\bullet-7\sigma_\bullet^2)/(8-4\sigma_\bullet)
&\quad\hbox{if $1/2\le\sigma_\bullet\le 4/7$},\\
9/10
&\quad\hbox{if $4/7\le\sigma_\bullet\le 3/5$,}\\
(\sigma_\bullet+3)/4
&\quad\hbox{if $3/5\le\sigma_\bullet<1$}.
\end{cases}
\]
\end{itemize}
\end{thm}

In particular, under GRH, by Theorem~\ref{thm:noSiegel} we have
\[
|\sfS_\mu({\mathbf D};M)|\le (A+B)N^{7/8+o(1)}\qquad(N\to\infty). 
\]

We expect that the condition that $\wp$ is a
\emph{linear} polynomial can be removed  with more effort; see Remark~\ref{rem:Conj1}. Thus, we propose the following conjecture.

\begin{conjecture}\label{conj:one}
Bounds similar to those of Theorem~\ref{thm:noSiegel} hold for any
bracket polynomial $\wp:\Z\to\Z$ of complexity $\delta$,
with constants that may depend only on~$\delta$.
\end{conjecture}

\subsection{Results over short intervals}
With only minor changes to the proof of Theorem~\ref{thm:main1-long},
we obtain the following result for short intervals.

\begin{thm}\label{thm:main1-short}
Let $C>0$, $\eps>0$, $N\ge 10$. Let ${\mathbf D}$ be
a collection of data~\eqref{eq:Data} such that
\[
\cA\defeq[1,A],\qquad\cB\defeq[1,B],\qquad\cN\defeq(N,N+H],
\]
where
\[
A,B,H\le N^{O_\delta(1)}\mand H\ge N^{5/8+\eps}.
\]
Assume further that $\wp$ is an ordinary polynomial in $\Z[X]$.
Then, the sums given by~\eqref{eq:SmuDM-defn}
satisfy the uniform bound
\[
 \sfS_\mu({\mathbf D};M)
\llsym{C,\eps,\delta}(A+B)H(\log N)^{-C}.
\]
\end{thm}

In the statement of Theorem~\ref{thm:main1-short}, the complexity
$\delta$ can be replaced by the degree $d$ of the polynomial
$\wp$ since clearly $\delta \ll d$.

We propose the following analogue of Conjecture~\ref{conj:one}.

\begin{conjecture}\label{conj:two}
The bounds of Theorem~\ref{thm:main1-short} hold for any
bracket polynomial $\wp:\Z\to\Z$ of complexity $\delta$
with implied constants that depend only on $C,\eps,\delta$.
In the same generality, bounds similar to those of
Theorem~\ref{thm:noSiegel} hold $($under the same conditions
on the zeros of Dirichlet $L$-functions$)$ with constants $c$ in~\eqref{eq:uni-one}
and $c(\sigma_\bullet)$ in~\eqref{eq:uni-two}
that depend only on $\delta$.
\end{conjecture}

{\large\section{Preliminaries}}

\subsection{Exponential sums with the M{\"o}bius function}
\label{sec:expsumsM}

Our work relies on bounds for sums $\sum\mu(n)\e(p(n))$
with bracket polynomials $p:\Z\to\R$, where
\[
\e(u)\defeq\exp(2 \pi i u)\qquad(u\in\R).
\]
The first result in this direction was obtained
in 1937 by  Davenport~\cite{Dav}
(see also~\cite[Thm.~13.10]{IwKow}), who showed that
\be\label{eq:Dav}
\sum_{n\le N}\mu(n)\e(\alpha n)\llsym{C} N(\log N)^{-C}
\ee
holds uniformly for $\alpha\in\R$. 
Note that the constant implied by~\eqref{eq:Dav} 
cannot be computed effectively
 due to the possible existence of Siegel zeros.
On the other hand, the independence of $\alpha$ in the
bound~\eqref{eq:Dav} is an attractive feature that plays an
essential role in this paper.

Davenport's theorem~\eqref{eq:Dav} has been extended in numerous ways.
The extension to bracket polynomials is due to
Green and Tao~\cite{GreenTao}; it follows as a consequence of
the main theorem~\cite[Thm.~1.1]{GreenTao} used in their
resolution of the M\"obius and Nilsequence conjecture, i.e.,
that the M\"obius function is strongly orthogonal to nilsequences.
More precisely we recall the following striking result,
which is due to Green and Tao~\cite[Thm.~5.2]{GreenTao}:
For any bracket polynomial
$p:\Z\to\R$ of complexity~$\delta$
and any \emph{Lipshitz function} $\Psi:[0,1]\to[-1,1]$, we have
\be\label{eq:GT-thm}
\sum_{n\le N}\mu(n)\Psi(\{p(n)\})\llsym{\delta,\Psi,C}
N(\log N)^{-C}.
\ee
We remark that the implied constant is ineffective.

The Green-Tao bound~\eqref{eq:GT-thm} clearly implies that
\be\label{eq:GT-long-gen}
\sum_{n\le N}\mu(n)\e(p(n))\llsym{\delta,C}N(\log N)^{-C}.
\ee
We use a slight variant of this bound in our proof of
Theorem~\ref{thm:main1-long}.

\begin{lemma}\label{lem:Davenportish Bound}
Let $C>0$, $N\ge 10$. For any $d\in\N$, $a\in\Z$,
and any bracket polynomial $p:\Z\to\R$ of complexity~$\delta$,
the following bound holds:
\be\label{eq:GT-long-APs}
\ssum{n\le N\\n\equiv a\bmod d}\mu(n)\e(p(n))\llsym{\delta,C}
N(\log N)^{-C}.
\ee
\end{lemma} 

\begin{proof}
Using orthogonality to detect the congruence condition, we have
\dalign{
\ssum{n\le N\\n\equiv a\bmod d}\mu(n)\e(p(n))
&=\sum_{n\le N}\mu(n)\e(p(n))\cdot\frac{1}{d}
\sum_{j\bmod d}\e(j(n-a)/d)\\
&=\frac{1}{d}\sum_{j\bmod d}\e(-ja/d)
\sum_{n\le N}\mu(n)\e(p(n)+jn/d).
}
Since each bracket polynomial $p(n)+jn/d$ has
complexity $O(\delta)$, every inner sum is 
$O_{\delta,C}(N(\log N)^{-C})$ by~\eqref{eq:GT-long-gen},
and the result follows.
\end{proof}

\begin{remark} 
We stress that the implied constant
in~\eqref{eq:GT-long-APs} is
independent of the arithmetic progression $a\bmod d$.
\end{remark}

For linear polynomials in $\Z[X]$, sharper discorrelation
bounds than~\eqref{eq:Dav} are available provided that 
the zeros of all Dirichlet $L$-functions are
sufficiently well-behaved. For example, Hajela and
Smith~\cite{HajSmith} have shown that the bound
\be\label{eq:HajSmi}
\sup\limits_{\alpha\in\R}
\biggl|\sum_{n\le N}\mu(n)\e(\alpha n)\biggr|\ll N\er^{-c\sqrt{\log N}}
\ee
holds (with some absolute effective constant $c$)
provided that Siegel zeros do not exist. Moreover,
they established stronger results (with a power savings in~$N$)
under the stronger assumption that every $L$-function $L(s,\chi)$
is nonzero in a half-plane $\sigma\ge\sigma_\bullet$
with some $\sigma_\bullet<1$. These
results were improved by Baker and Harman~\cite{BaHa}, 
and further by Zhang~\cite{Zhang}; hence, it is known
(under the same hypothesis) that
\be\label{eq:BakHar}
\sup\limits_{\alpha\in\R}
\biggl|\sum_{n\le N}\mu(n)\e(\alpha n)\biggr|\le N^{b(\sigma_\bullet)+o(1)}\qquad(N\to\infty),
\ee
where
\be\label{eq:bsigdefn}
b(\sigma_\bullet)\defeq\begin{cases}
(8\sigma_\bullet-7\sigma_\bullet^2)/(4-2\sigma_\bullet)
&\quad\hbox{if $1/2\le\sigma_\bullet\le 4/7$,}\\
4/5
&\quad\hbox{if $4/7\le\sigma_\bullet\le 3/5$,}\\
(\sigma_\bullet+1)/2
&\quad\hbox{if $3/5\le\sigma_\bullet<1$.}
\end{cases}
\ee
In particular, under the GRH, we have
\[
\sup\limits_{\alpha\in\R}
\biggl|\sum_{n\le N}\mu(n)\e(\alpha n)\biggr|\le N^{3/4+o(1)}\qquad(N\to\infty)
\]
(we remark that in this case the results of Baker and Harman~\cite{BaHa}, 
and of  Zhang~\cite{Zhang} coincide).

Arguing as in the proof of Lemma~\ref{lem:Davenportish Bound},
using~\eqref{eq:HajSmi} or~\eqref{eq:BakHar} in place of~\eqref{eq:GT-long-gen}, 
we have the following two estimates.

\begin{lemma}
\label{lem:disc-noSiegel}
Suppose that Siegel zeros do not exist for Dirichlet $L$-functions.
Then, for $N\ge 10$, $d\in\N$, and $a\in\Z$, we have
\[
\sup\limits_{\alpha\in\R}
\biggl|\ssum{n\le N\\n\equiv a\bmod d}\mu(n)\e(\alpha n)\biggr|
\ll N\er^{-c\sqrt{\log N}}
\]
for some effectively computable constant $c>0$.\end{lemma} 

\begin{lemma}
\label{lem:disc-zero strip}
Suppose that every Dirichlet $L$-function $L(s,\chi)$
is nonzero in the half-plane $\sigma\ge\sigma_\bullet$
for some $\sigma_\bullet<1$.
Then, for $N\ge 10$, $d\in\N$, and $a\in\Z$, we have
\[
\sup\limits_{\alpha\in\R}
\biggl|\ssum{n\le N\\n\equiv a\bmod d}\mu(n)\e(\alpha n)\biggr|
\le N^{b(\sigma_\bullet)+o(1)}\qquad(N\to\infty),
\]
where $b(\sigma_\bullet)$ is given by~\eqref{eq:bsigdefn}.
\end{lemma} 

\begin{remark}\label{rem:Conj1}
In Theorem~\ref{thm:noSiegel}
the hypothesis that $\wp$ is a linear polynomial in $\Z[X]$ is
only needed in order to apply Lemmas~\ref{lem:disc-noSiegel}
and~\ref{lem:disc-zero strip} in their proofs.
To prove Conjecture~\ref{conj:one}, it suffices to
establish analogs of~\eqref{eq:HajSmi} and~\eqref{eq:BakHar}
for arbitrary bracket polynomials $\wp:\Z\to\Z$. It is very plausible that 
the ideas and results of  Green and Tao~\cite{GreenTao} can be used
to prove these results.
\end{remark}

Zhan~\cite{Zhan} studied the discorrelation between $\mu$ and real additive characters over short intervals. He showed that
\be\label{eq:Zhan}
\sup\limits_{\alpha\in\R}
\biggl|\sum_{N<n\le N+H}\mu(n)\e(\alpha n)\biggr|\llsym{\eps, A}
H(\log N)^{-C}
\ee
holds in the range $N^{5/8+\eps}\le H\le N$. More recently,
the discorrelation of $\mu$ with arbitrary polynomial phases
over short intervals has been studied by Matom\"aki, 
Shao, Tao, and Ter\"av\"ainen~\cite{MSTT} as part of their
 program to explore correlations of arithmetic 
functions $f:\N\to\C$ with arbitrary nilsequences in short intervals.
In particular, their result~\cite[Cor.~1.3\,$(i)$]{MSTT} implies that
if $N^{5/8+\eps}\le H\le N^{O_\delta(1)}$, and $P:\Z\to\R$ is any polynomial
of degree $\delta\ge 0$, then for all $A>0$, one has
\be\label{eq:MSTT}
\sum_{N<n\le N+H}\mu(n)\e(P(n))\llsym{C,\eps,\delta}
H(\log N)^{-C}.
\ee
Note that in the formulation of~\cite[Cor.~1.3\,$(i)$]{MSTT},
the range $N^{5/8+\eps}\le H\le N^{1-\eps}$ is assumed for convenience,
but the result is easily extended to cover the range
$N^{1-\eps}\le H\le N^{O_\delta(1)}$ by splitting long sums into
shorter ones; see Remark~1.2 in~\cite{MSTT}. Since the
upper bound~\eqref{eq:MSTT} is the same for all polynomials of a fixed
degree $\delta$, the bound~\eqref{eq:Zhan} of Zhan~\cite{Zhan} follows
(in the same extended range)
by taking $P(x)\defeq\alpha x$ 
and letting $\alpha$ vary over~$\R$.
We remark that, for smaller $H$ above $N^{3/5+\eps}$,~\cite[Cor.~1.3\,$(iii)$]{MSTT} 
gives a nontrivial bound;
however, the savings over the trivial bound is only $(\log N)^{1/3}$,
which is insufficient for our purposes.

Using~\eqref{eq:MSTT}  and adapting 
the proof of Lemma~\ref{lem:Davenportish Bound}, 
we have the following result on discorrelation of $\mu$ with polynomial phases in 
short arithmetic progressions.

\begin{lemma}
\label{lem:Davenport  Bound-Short}
Let $N\ge 2$, and let
$N^{5/8+\eps}\le H\le N^{O_\delta(1)}$ for some
$\eps>0$. Let $\wp:\Z\to\R$ be an s an ordinary polynomial in $\Z[X]$ of degree 
$\delta\ge 0$.
Then, for all $A>0$, $d\in\N$, $a\in\Z$, the following bound holds:
\[
\ssum{N<n\le N+H\\n\equiv a\bmod d}\mu(n)\e(P(n))\llsym{C,\eps,\delta}
H(\log N)^{-C}.
\]
\end{lemma} 

\subsection{Sums with the divisor function} 
From now on, we use $\tau(n)$ to denote the number of positive
integer divisors of nonzero integer; this differs from the notation
used earlier in~\eqref{eq:Rama}. 

We need some elementary bounds on sums with $\tau$. 
First, we recall a classical result on mean values of powers of
the divisor function. For any fixed positive integer $k$, one has
\be\label{eq:Div}
\sum_{n\le x}\tau(n)^k 
\llsym{k} x(\log x)^{2^k-1}\qquad(x\ge 2),
\ee 
where the implied constant depends only on $k$; see, e.g.,
Iwaniec and Kowalski~\cite[eqn.~(1.81)]{IwKow} for a more precise statement.
Using~\eqref{eq:Div} and partial summation, we immediately derive
the estimates:
\be\label{eq:SumDiv1/2}
\sum_{n\le x}\frac{\tau(n)^k}{n^{1/2}}
\ll x^{1/2}(\log x)^{2^k-1}\qquad(x\ge 2)
\ee
and
\be\label{eq:SumDiv3/2}
\sum_{n>x}\frac{\tau(n)^k}{n^{3/2}}
\ll x^{-1/2}(\log x)^{2^k-1}\qquad(x\ge 2).
\ee
We also need the following simple bound: 
\be\label{eq:SumGCD}
\sum_{n\le x}\frac{\sqrt{\gcd(m,n)}}{n}\ll\tau(m)\log x
\qquad(m\in\N,~x\ge 2),
\ee
To prove this, we simply observe that
\[
\sum_{n\le x}\frac{\sqrt{\gcd(m,n)}}{n}
\le\sum_{d\mid m}d^{1/2}\ssum{n\le x\\d\mid n}\frac{1}{n}
\le\sum_{d\mid m}d^{-1/2}\ssum{n\le x/d}\frac{1}{n}
\]
and bound the final double sum crudely. 

{\large\section{Proof of Theorem~\ref{thm:main1-long}}}

\subsection{Preliminary split} 
Writing $S$ for $\sfS_\mu({\mathbf D};M)$, we have
by the orthogonality of exponential functions:
\[
S=
\int_0^1\ssum{n_1\le A\\n_2\le B\\n_3\le N}
{\tt u}_{n_1}{\tt v}_{n_2}\,\mu(n_1n_2n_3)
\e(\alpha(f(n_1)+g(n_2)+\wp(n_3)-M))\,d\alpha.
\]
Taking into account that
\be\label{eq:mobius madness}
\mu(mn)=\begin{cases}
\mu(m)\mu(n)&\quad\hbox{if $\gcd(m,n)=1$},\\
0&\quad\hbox{otherwise},
\end{cases}
\ee
it follows that
\dalign{
S&=\int_0^1\e(-\alpha M)
\ssum{n_1\le A\\n_2\le B}
{\tt u}_{n_1}{\tt v}_{n_2}\,\mu(n_1n_2)
\e(\alpha f(n_1)+\alpha g(n_2))\\
&\qquad\qquad\qquad\qquad\qquad\times
\ssum{n_3\le N\\\gcd(n_1n_2,n_3)=1}
\mu(n_3)\e(\alpha \wp(n_3))\,d\alpha.
}
Using the inclusion-exclusion principle to detect coprimality,
we can write
\be\label{eq:eureka}
S=\sum_{d\le N}\mu(d)I(d)
=\sum_{d\le D}\mu(d)I(d)+\sum_{D<d\le N}\mu(d)I(d)
\ee
with any real $D\in[2,N]$, where
\dalign{
I(d)&= \int_0^1\e(-\alpha M)
\ssum{n_1\le A\\n_2\le B\\d\mid n_1n_2}
{\tt u}_{n_1}{\tt v}_{n_2}\,\mu(n_1n_2)
\e(\alpha f(n_1)+\alpha g(n_2))\\
&\qquad\qquad\qquad\qquad\qquad\times
\ssum{n_3\le N\\d\mid n_3}
\mu(n_3)\e(\alpha \wp(n_3))\,d\alpha.
}
In what follows, we focus on bounding the two
sums on the right side of~\eqref{eq:eureka}.

\subsection{The sum over small $d$}
For each $d\le D$, Lemma~\ref{lem:Davenportish Bound} (with $a=0$
and $2C+4+1/\eps$ in place of $C$) provides the bound
\be\label{eq:Id Jd}
I(d) \llsym{C,\eps,\delta}N(\log N)^{-2C-4-1/\eps}\cdot J(d) 
\ee  
for any $C>0$, where 
\[
J(d)\defeq\int_0^1\Biggl|
\ssum{n_1\le A\\n_2\le B\\d\mid n_1n_2}
{\tt u}_{n_1}{\tt v}_{n_2}\,\mu(n_1n_2)\e(\alpha f(n_1)+\alpha g(n_2))\Biggr|\,d\alpha.
\]
Taking into account that $\mu(n_1n_2)=0$ unless $\gcd(n_1,n_2)=1$,  
the inner sum splits into subsums that are parametrized by ordered
pairs $(d_1,d_2)$ of coprime positive integers such that
$d_1d_2=d$. Using the triangle inequality, we have
\be\label{eq:Jd Kdd}
J(d)\le\ssum{d_1d_2=d\\\gcd(d_1,d_2)=1} K(d_1, d_2), 
\ee
where 
\[
K(d_1,d_2)\defeq\int_0^1\Biggl|
\ssum{n_1\le A\\n_2\le B\\d\mid n_1n_2}
{\tt u}_{n_1}{\tt v}_{n_2}\,\mu(n_1n_2)\e(\alpha f(n_1)+\alpha g(n_2))\Biggr|\,d\alpha.
\]
By~\eqref{eq:mobius madness} and the
inclusion-exclusion principle, we have
\[
\mu(n_1n_2)=\mu(n_1)\mu(n_2)\sum_{e\mid\gcd(n_1,n_2)}\mu(e)
=\mu(n_1)\mu(n_2)\ssum{
e\mid n_1,\,e\mid n_2}\mu(e). 
\]
Since $d_j$ and $e$ both divide $n_j$ if and only if the least
common multiple $\lcm[d_j,e]$ divides $n_j$, we see that
\be\label{eq:Kdd Ldde}
K(d_1,d_2)\le\sum_{e\le \max\{A,B\}} L(d_1,d_2;e),
\ee
where
\[
L(d_1,d_2;e)\defeq
\int_0^1\Biggl|\ssum{n_1\le A\\ \lcm[d_1,e]\mid n_1}
{\tt u}_{n_1}\,\mu(n_1)\e(\alpha f(n_1))
\ssum{n_2\le B\\ \lcm[d_2,e]\mid n_2}
{\tt v}_{n_2}\,\mu(n_2)\e(\alpha g(n_2))\Biggr|\,d\alpha.
\]
By the Cauchy inequality,
\begin{align*}L(d_1,d_2;e)^2 \le 
\int_0^1\Biggl|\ssum{n_1\le A\\ \lcm[d_1,e]\mid n_1} &
{\tt u}_{n_1}\,\mu(n_1)\e(\alpha f(n_1))\Biggr|^2\,d\alpha\\
&  \times\int_0^1\Biggl|\ssum{n_2\le B\\ \lcm[d_2,e]\mid n_2}
{\tt v}_{n_2}\,\mu(n_2)\e(\beta g(n_2))\Biggr|^2\,d\beta.
\end{align*} 
Recalling that $f:\cA\cap{\rm Supp}({\tt u})\to\Z$
is \emph{injective}, and $\|{\tt u}\|_\infty\le 1$,
it follows that
\[
\int_0^1\Biggl|\ssum{n_1\le A\\ \lcm[d_1,e]\mid n_1}
{\tt u}_{n_1}\,\mu(n_1)\e(\alpha f(n_1))\Biggr|^2\,d\alpha
=\ssum{n_1\le A\\ \lcm[d_1,e]\mid n_1}|{\tt u}_{n_1}|^2\mu(n_1)^2
\le\frac{A}{\lcm[d_1,e]},
\]
and similarly,
\[
\int_0^1\Biggl|\ssum{n_2\le B\\ \lcm[d_2,e]\mid n_2}
{\tt v}_{n_2}\,\mu(n_2)\e(\beta g(n_2))\Biggr|^2\,d\alpha
\le\frac{B}{\lcm[d_2,e]};
\]
consequently, we have
\[
L(d_1,d_2;e)^2\le\frac{AB}{\lcm[d_1,e]\lcm[d_2,e]}
=\frac{AB\gcd(d_1d_2,e)}{d_1d_2e^2}.
\]
Using~\eqref{eq:SumGCD}, ~\eqref{eq:Kdd Ldde}, and the bound
$\log\max\{A,B\}\le(\log N)^{1/\eps}$, we have
\[
K(d_1,d_2)\le\frac{\sqrt{AB}}{\sqrt{d_1d_2}}
\sum_{e\le\max\{A,B\}}\frac{\sqrt{\gcd(d_1d_2,e)}}{e}
\llsym{\delta}\sqrt{AB}(\log N)^{1/\eps}\cdot\frac{\tau(d_1d_2)}{\sqrt{d_1d_2}}.
\]
In view of~\eqref{eq:Jd Kdd}, we have
\be\label{eq:Jd bound}
J(d)\llsym{\delta} \sqrt{AB}(\log N)^{1/\eps}
\ssum{d_1d_2=d\\\gcd(d_1,d_2)=1}
\frac{\tau(d_1d_2)}{\sqrt{d_1d_2}}
\le \sqrt{AB}(\log N)^{1/\eps}\cdot\frac{\tau(d)^2}{\sqrt{d}}. 
\ee
Finally, using~\eqref{eq:Id Jd} followed by~\eqref{eq:SumDiv1/2},
we derive that
\dalign{
\sum_{d\le D}\mu(d)I(d)
& \llsym{C,\eps,\delta}\sqrt{AB}\,N(\log N)^{-2C-4}
\sum_{d\le D}\frac{\tau(d)^2}{\sqrt{d}}\\
&\ll \sqrt{AB}\,N (\log N)^{-2C-4}D^{1/2}(\log D)^3;
}
therefore, since $D\le N$,
\be\label{eq:sum-small-d}
\sum_{d\le D}\mu(d)I(d)
 \llsym{C,\eps,\delta}
\sqrt{AB}\,N(\log N)^{-2C-1}D^{1/2}.
\ee

\subsection{The sum over large $d$} 
\label{sec:large d}
Now suppose $D<d\le N$.
Trivially, $|I(d)|$ does not exceed the number of
tuples
\[
\qquad(n_1,n_2,n_3)\in\cA_{{\tt u}}\times\cB_{{\tt v}}\times \cN 
\qquad\bigl(\cA_{{\tt u}}\defeq\cA\cap{\rm Supp}({\tt u}),~
\cB_{{\tt v}}\defeq\cA\cap{\rm Supp}({\tt v})\bigr)
\]
for which
\be\label{eq:d | n1n2 n3}
f(n_1)+g(n_2)+\wp(n_3)=M,\qquad d\mid n_1n_2,\qquad d \mid n_3.
\ee
Since $\cN=[1,N]$, there are at most $Nd^{-1}$ possible values
of $n_3$. Furthermore, there is a divisor $d_0\mid d$ such that
$d_0\ge d^{1/2}$, and $d_0$ divides either $n_1$ or $n_2$. 
Since $\cA_{{\tt u}}\subseteq[1,A]$ and
$\cB_{{\tt v}}\subseteq[1,B]$, 
either $n_1$ or $n_2$ takes at most
$\max\{A,B\}d_0^{-1}\ll(A+B)d^{-1/2}$ distinct values,
after which the last remaining variable is uniquely determined (using
the injectivity of $f$ and $g$ on $\cA_{{\tt u}}$ and
$\cB_{{\tt v}}$, respectively). Putting everything together,
we have
\[
I(d)\ll (A+B)N\frac{\tau(d)}{d^{3/2}}.
\]
Then, using~\eqref{eq:SumDiv3/2}, we derive the bound
\be\label{eq:sum-large-d}
\sum_{d>D}\mu(d)I(d)\ll (A+B)ND^{-1/2}\log N.
\ee

\subsection{Concluding the proof} 
Substituting~\eqref{eq:sum-small-d} and~\eqref{eq:sum-large-d}  
into~\eqref{eq:eureka}, we have (since $\sqrt{AB}\le A+B$):
\[
S \llsym{C,\eps,\delta}(A+B)N(D^{1/2}(\log N)^{-2C-1}
+D^{-1/2}\log N).
\]
We take $D\defeq(\log N)^{2C+2}$,
which balances the two terms
and also satisfies the condition $D\in[2,N]$,
and the theorem follows.

{\large\section{Proof of Theorem~\ref{thm:noSiegel}}}\label{sec:zeros}

We proceed essentially as in the proof of
Theorem~\ref{thm:main1-long}, except that in place
of Lemma~\ref{lem:Davenportish Bound}, we use either
Lemma~\ref{lem:disc-noSiegel} or Lemma~\ref{lem:disc-zero strip},
leading (respectively) to 
\begin{align*}
I(d)&\ll N\er^{-10c\sqrt{\log N}}\cdot J(d),\\
\bigl|I(d)\bigr|&\le N^{b(\sigma_\bullet)+o(1)}\cdot J(d)\qquad(N\to\infty). 
\end{align*}
These bounds replace~\eqref{eq:Id Jd}.
Note that these results rely on the fact that $\wp$ is a linear polynomial (this hypothesis is
used only at this point in the proof). Next,
using~\eqref{eq:Jd bound} to bound each $J(d)$,
and summing over $d\le D$, we have (respectively)
\dalign{
\sum_{d\le D}\mu(d)I(d)
&\ll\sqrt{AB}\,N\er^{-10c\sqrt{\log N}}D^{1/2},\\
\biggl|\sum_{d\le D}\mu(d)I(d)\biggr|
&\le\sqrt{AB}\,N^{b(\sigma_\bullet)+o(1)}D^{1/2}
\qquad(N\to\infty).
}
Combining with~\eqref{eq:sum-large-d}, we have
(respectively)
\dalign{
S&\llsym{C,\eps,\delta}(A+B)N(\er^{-10c\sqrt{\log N}}D^{1/2}
+D^{-1/2}\log N),\\
S&\llsym{C,\eps,\delta}(A+B)(N^{b(\sigma_\bullet)+o(1)}D^{1/2}
+ND^{-1/2}\log N)\qquad(N\to\infty).\\
}
To prove~\eqref{eq:uni-one}, we choose
$D\defeq\er^{5c\sqrt{\log N}}$ to balance everything,
whereas for~\eqref{eq:uni-two},
we take $D\defeq N^{1-b(\sigma_\bullet)}$ (and thus,
$c(\sigma_\bullet)\defeq\frac12(b(\sigma_\bullet)+1)$).
In either case, the result follows by an easy computation.

\bigskip

{\large \section{Proof of Theorem~\ref{thm:main1-short}}}

Observe that for $H>N$ the bound follows trivially from 
Theorem~\ref{thm:main1-long}. Hence we can assume that 
$H\le N$ in what follows.

We proceed as in the proof of Theorem~\ref{thm:main1-long}, 
 using the slightly altered notation
\[
S\defeq
\int_0^1\ssum{n_1\le A\\n_2\le B\\N<n_3\le N+H}
{\tt u}_{n_1}{\tt v}_{n_2}\,\mu(n_1n_2n_3)
\e(\alpha(f(n_1)+g(n_2)+\wp(n_3)-M))\,d\alpha
\]
and
\dalign{
I(d)&\defeq\int_0^1\e(-\alpha M)
\ssum{n_1\le A\\n_2\le B\\d\mid n_1n_2}
{\tt u}_{n_1}{\tt v}_{n_2}\,\mu(n_1n_2)\e(\alpha f(n_1)+\alpha g(n_2))\\
&\qquad\qquad\qquad\qquad\qquad\times
\ssum{N<n_3\le N+H\\d\mid n_3}
\mu(n_3)\e(\alpha \wp(n_3))\,d\alpha. 
}
Let $D\in[2,D_\infty]$ (to be specified later), where
\[
D_\infty\defeq\max\{A,B,N+H\}\le N^{O_\delta(1)}.
\]

For small $d\le D$, instead of
Lemma~\ref{lem:Davenportish Bound}, we use 
Lemma~\ref{lem:Davenport  Bound-Short}
 with $C$ replaced by
\be\label{eq:Qdefn}
Q\defeq 4C+5\qquad(C>0),
\ee 
which leads to the bound (cf.\,\eqref{eq:Id Jd}):
\[
I(d) \llsym{C,\eps,\delta}H(\log N)^{-Q}\cdot J(d),
\]
where $J(d)$ is defined as before. Proceeding as before,
but using the bound $\max\{A,B\}\le N^{O_\delta(1)}$
in place of $\log\max\{A,B\}\le(\log N)^{1/\eps}$, we get that
\[
J(d)\llsym{\delta}\sqrt{AB}(\log N)
\cdot\frac{\tau(d)^2}{\sqrt{d}},
\]
and it follows that 
\be\label{eq:sum-small-d tilde}
\sum_{d\le D}\mu(d) I(d)
\llsym{C,\eps,\delta}
\sqrt{AB}\,H(\log N)^{4- Q}D^{1/2}.
\ee

Next, we need an analogue of~\eqref{eq:sum-large-d}.
Adapting the argument 
in~\S\ref{sec:large d}, we observe that~\eqref{eq:d | n1n2 n3} 
implies that for each $n_3\le N+H$,
there are no more than $Hd^{-1}+1$ possible values of $n_3$
if $d\le N+H$, and no such values 
for larger $d$. Hence, as before, we deduce the bound
\[
I(d)\ll (A+B)(Hd^{-1}+1)\tau(d)d^{-1/2}.
\]
for any $d \le N+H$, and $I(d) = 0$, otherwise.
Using both~\eqref{eq:SumDiv1/2} and~\eqref{eq:SumDiv3/2},
and taking into account that $H\le N$,
it follows that
\dalign{
\sum_{d>D}\mu(d)I(d)
&\ll (A+B)H\sum_{d>D}\frac{\tau(d)}{d^{3/2}}
+(A+B)\sum_{d\le N+H}\frac{\tau(d)}{d^{1/2}}\\
&\ll  (A+B)HD^{-1/2}\log N+(A+B)N^{1/2}\log N,
} 
and therefore
\be\label{eq:sum-large-d tilde}
\sum_{d>D}\mu(d)I(d)\ll 
(A+B)(HD^{-1/2}+N^{1/2})\log N.
\ee

In view of~\eqref{eq:eureka}, we have upon combining 
\eqref{eq:sum-small-d tilde} and~\eqref{eq:sum-large-d tilde}:
\[
S\llsym{C,\eps,\delta}
\sqrt{AB}\,H(\log N)^{4- Q}D^{1/2}
+(A+B)(HD^{-1/2}+N^{1/2})\log N.
\]
We now choose
\[
D\defeq\frac{(A+B)(\log N)^{Q-3}}{\sqrt{AB}},
\]
which balances two terms in the above bound
and also satisfies the condition that $D\in[2,D_\infty]$. 
As the choice of $Q$ in~\eqref{eq:Qdefn}
implies that $(Q-3)/2-1 = 2C$, we conclude that
\be\label{eq:Salth}
S\llsym{C,\eps,\delta}
(A^{3/4}B^{1/4}+A^{1/4}B^{3/4})H(\log N)^{-2C}
+(A+B)N^{1/2}\log N.
\ee
Now we are in a position to finish the proof.

First, note that if $B\le A(\log N)^{-C}$, 
then the trivial bound~\eqref{eq:triv-bd} implies that
\[
|S|\le BH\le AH(\log N)^{-C},
\]
and the theorem holds in this case.
On the other hand, if $A\le B(\log N)^C$,
then by~\eqref{eq:Salth} we have
\[
S\llsym{C,\eps,\delta}BH(\log N)^{-5C/4}
+B N^{1/2}(\log N)^{C+1}.
\]
Both terms on the right side are 
$O_{C,\eps,\delta}(BH(\log N)^{-C})$, with
the bound on the second term being a consequence of
the conditions impose on $H$ and $N$ in the theorem:
\[
B N^{1/2}(\log N)^{C+1}\llsym{C,\eps} BN^{5/8+\eps}(\log N)^{-C}
\le BH(\log N)^{-C}.
\]
This completes the proof.

\bigskip

{\large\section{Applications}\label{sec:applications}}

\subsection{Primes (I)}\label{sec:applPrimesI}

Let $\P$ denote the set of primes.
Let $C>0$, $x\ge 10$, and consider the
collection ${\mathbf D}$ comprised of the following data:
\begin{itemize}
\item Intervals $\cA,\cB,\cN\subseteq\N$ given by
\[
\cA=\cB\defeq[1,A],\qquad\cN\defeq[1,x]
\qquad\text{with}\quad A\le \er^{(\log x)^C};
\] 
\item The polynomial $\wp(X)=-X$;
\item Weights ${\tt u}=\{{\tt u}_n\}$
and ${\tt v}=\{{\tt v}_n\}$ given by
\[
{\tt u}_n={\tt v}_n\defeq\begin{cases}
1&\quad\hbox{if $n\in\P$},\\
0&\quad\hbox{otherwise}. 
\end{cases}
\]
\item Injective functions $f,g:\cA\cap\P\to\Z$.
\end{itemize}
Using $p,q$ to denote primes, Theorem~\ref{thm:main1-long}
(with $M=0$) yields the bound
\[
\ssum{p,q\le A\\1\le f(p)+g(q)\le x}
\mu\bigl(pq\,(f(p)+g(q))\bigr)
\llsym{\eps,C} Ax(\log x)^{-C}.
\]
Moreover, by making suitable adjustments to
the collection ${\mathbf D}$, we can establish the general bound 
\[
\ssum{a<p\le A,~b<q\le A\\1\le f(p)+g(q)\le x}
\mu\bigl((p-a)(q-b)(f(p)+g(q))\bigr)
\llsym{\eps,C} Ax(\log x)^{-C}
\]
with any $a,b\in\Z$ provided that $A\ge 2\max\{|a|,|b|\}$.
The choice $f(n)=g(n)\defeq n$ for all $n$
leads to the bound
\[
\ssum{p>a,\,q>b\\p+q\le x}
\mu\bigl((p-a)(q-b)(p+q)\bigr)
\llsym{\eps,C} x^2(\log x)^{-C}.
\]
In particular,
\be\label{eq:DBowie1a}
\ssum{
p+q\le x}\mu(p+q)\llsym{C}x^2(\log x)^{-C}
\ee
and
\be\label{eq:DBowie1b}
\ssum{
p+q\le x}\mu\bigl((p-1)(q-1)(p+q)\bigr)
\llsym{C}x^2(\log x)^{-C}.
\ee
Although we have not found any alternative approaches to proving~\eqref{eq:DBowie1b}, it is worth mentioning that the first 
bound~\eqref{eq:DBowie1a} can also be established directly
using a different (albeit more complicated) approach based on
results from the circle method combined with
standard techniques in analytic number theory.
In fact, applying partial summation,
we see that~\eqref{eq:DBowie1a} is equivalent to 
\be\label{eq:DBowie2}
\ssum{m\le x}\mu(m)R_1(m)
=\ssum{
p+q\le x}(\log p)(\log q)\mu(p+q)
\llsym{C}x^2(\log x)^{-C},
\ee
where, as in~\cite[eqn.~(3.16)]{Vau}, we define 
\[
R_1(m)\defeq\ssum{
p+q=m}(\log p)(\log q).
\] 
The function $R_1(m)$ is well known in the theory
of the binary Goldbach problem; see, e.g., Vaughan~\cite[\S3.2]{Vau}.
Using~\cite[Thm.~3.7]{Vau} and the Cauchy inequality, we deduce
the estimate
\be\label{eq:DBowie3}
\ssum{m\le x}\mu(m)R_1(m)
=\ssum{m\le x}\mu(m)m\fS_1(m)
+O_C\bigl(x^2(\log x)^{-C}\bigr),
\ee
where $\fS_1(m)$ is the singular series given by
\[
\fS_1(m)\defeq\prod_{p\,\nmid\, m}(1-(p-1)^{-2})
\prod_{p\mid m}(1+(p-1)^{-1}); 
\]
see~\cite[eqn.~(3.26)]{Vau}.
One verifies that the associated Dirichlet series
\[
D(s)\defeq\sum_{n=1}^\infty\frac{\mu(n)n\fS_1(n)}{n^s}
\]
converges absolutely in the half-plane $\sigma>2$, and
$D(s)=\zeta(s-1)^{-1}F(s)$ near the point $s=2$
for some function $F(s)$ that is analytic and nonzero at $s=2$.
Perron's formula (see~\cite[Thm.~5.2 and Cor.~5.3]{MontVau}), 
yields the bound
\[
\ssum{m\le x}\mu(m)m\fS_1(m)\llsym{C}x^2(\log x)^{-C}.
\]
Combining this with~\eqref{eq:DBowie3}, we obtain~\eqref{eq:DBowie2}.

\begin{remark} It is straightforward to verify the lower bound
$\sum_{m\le x}R_1(m)\gg  x^2$. Thus, in view of~\eqref{eq:DBowie2}, we can say that
$R_1(m)$ is orthogonal to $\mu(m)$
in the sense of  Sarnak {\rm\cite{Sarnak}}.
\end{remark}

\subsection{Primes (II)}\label{sec:applPrimesII}
As another example, let $p_j$ denote the $j$-th prime
for each $j\in\N$, and let $\pi(x)$ be the prime counting function.
Let ${\mathbf D}$ consist of the data:
\begin{itemize}
\item Intervals $\cA,\cB, \cN\subseteq\N$ given by
\[
\cA=\cB\defeq[1,\pi(x)],\qquad \cN\defeq[1,2x].
\]
\item The polynomial $\wp(X)=-X$;
\item Weights ${\tt u}=\{{\tt u}_n\}$
and ${\tt v}=\{{\tt v}_n\}$ given by the M\"obius function:
\[
{\tt u}_n={\tt v}_n=\mu(n)\qquad(n\in\N);
\]
\item Injective functions $f,g:\cA\to\Z$
defined by $f(k)=g(k)\defeq p_k$.
\end{itemize}
Then Theorem~\ref{thm:main1-long} (with $M=0$) implies
the following bound:
\[
\ssum{p_k,p_\ell\le x\\\gcd(k\,\ell,\,p_k+p_\ell)=1}
\mu^2(k\ell)\,\mu(p_k+p_\ell)\llsym{C} x^2(\log x)^{-C}
\]
Moreover, as $\wp$ is a linear polynomial in $\Z[x]$,
the stronger bounds of Theorem~\ref{thm:noSiegel} are available 
if the zeros of Dirichlet $L$-functions are sufficiently well-behaved.

\subsection{Squarefrees}

Let $C>0$, $x\ge 10$, and suppose ${\mathbf D}$ consists of the data:
\begin{itemize}
\item Intervals $\cA,\cB, \cN\subseteq\N$ given by
\[
\cA=\cB\defeq[1,A],\qquad\cN\defeq[1,x]
\qquad\text{with}\quad A\le \er^{(\log x)^C};
\]
\item The polynomial $\wp(X)=-X$;
\item For fixed $\nu\in\{1,2\}$, the weights ${\tt u}=\{{\tt u}_n\}$
and ${\tt v}=\{{\tt v}_n\}$ are given by
\[
{\tt u}_n={\tt v}_n\defeq\mu(n)^{\nu+1}\qquad(n\in\N);
\]
\item The injective functions $f,g:\cA\to\Z$
defined by $f(k)=g(k)\defeq k$.
\end{itemize}
Using $k,\ell$ to denote natural numbers,
Theorem~\ref{thm:main1-long} (with $M=0$) shows that
\[
\ssum{k,\ell\le A\\x<k+\ell\le x+H}
\mu^\nu(k\ell)\,\mu(k+\ell)
=\ssum{k,\ell\le A\\x<k+\ell\le x+H}
{\tt u}_k\,{\tt v}_\ell\,\mu(k\ell(k+\ell))
\llsym{\eps,C}A H (\log x)^{-C}.
\]
In particular, for $A\defeq x$ we have 
\[
\ssum{
k+\ell\le x}
\mu^\nu(k\ell)\,\mu(k+\ell)
\llsym{C}x^2(\log x)^{-C}\qquad(\nu=1\text{~or~}2).
\]

\subsection{Beatty sequences}
Fix $\alpha,\beta\in\R$, where $\alpha>0$ is irrational.
Let $C>0$, $x\ge 10$, and suppose ${\mathbf D}$ consists of the data:
\begin{itemize}
\item Intervals $\cA,\cB, \cN\subseteq\N$ given by
\[
\cA=\cB\defeq[1,x],\qquad\cN\defeq\bigl[1,\fl{\alpha x+\beta}\bigr];
\]
\item The bracket polynomial $\wp(X)=-\fl{\alpha X+\beta}$;
\item Trivial weights ${\tt u}_n={\tt v}_n\defeq 1$ for all $n\in\N$;
\item The injective functions $f,g:\cA\to\Z$
defined by $f(k)=g(k)\defeq k$.
\end{itemize}
Using $k,\ell$ to denote natural numbers,
Theorem~\ref{thm:main1-long} (with $M=0$) shows that
\[
\ssum{k,\ell\le x,\,n\ge 1\\k+\ell=\fl{\alpha n+\beta}}
\mu(k\ell n)\llsym{C}x^2(\log x)^{-C}
\]

\subsection{Quadratic residues} Let us write $\reduceit{n}{m}$ as 
shorthand for $n\bmod m$, the remainder when $n$ is divided by $m$.
Thus, $\reduceit{n}{m}\defeq m\{n/m\}$, and $0\le\reduceit{n}{m}<m$.
For any bracket polynomial $\wp:\Z\to\Z$ of complexity $\delta$,
let $\reduceit{\wp}{m}:\Z\to\Z$ be defined by
\[
n\mapsto\reduceit{\wp(n)}{m}=m\{\wp(n)/m\}\qquad(n\in\Z).
\]
Then $\reduceit{\wp}{m}$ is a bracket polynomial of
complexity $\delta+O(1)$.

Fix a large integer $m$. Let $C>0$, $x\ge 10$,
and suppose ${\mathbf D}$
consists of the data:
\begin{itemize}
\item Intervals $\cA,\cB, \cN\subseteq\N$ given by
\[
\cA=\cB\defeq[1,m],\qquad\cN\defeq[1,2m];
\]
\item The bracket polynomial $\wp(X)=-\reduceit{X^2}{m}$;
\item Trivial weights ${\tt u}_n={\tt v}_n\defeq 1$ for all $n\in\N$;
\item The injective functions $f,g:\cA\to\Z$
defined by $f(k)=g(k)\defeq k$.
\end{itemize}
Using $a,b,c$ to denote natural numbers,
Theorem~\ref{thm:main1-long} (with $M=0$ and $M=m$)
provides the bounds
\[
\ssum{a,b,c\le m\\a+b=\reduceit{c^2}{m}}\mu(abc)
\llsym{C}m^2(\log m)^{-C},\qquad
\ssum{a,b,c\le m\\a+b=\reduceit{c^2}{m}+m}\mu(abc)
\llsym{C}m^2(\log m)^{-C},
\]
which can be combined into the single bound
\[
\ssum{a,b,c\le m\\a+b\equiv c^2\,({\rm mod}\,m)}\mu(abc)
\llsym{C}m^2(\log m)^{-C}.
\]

\subsection{Sums with one small variable}
\label{sec:small var} 
Our underlying approach also applies to cases where one of the 
variables is restricted to a short interval. We illustrate this in 
a concrete setting with the sums $S_\mu(H,N)$ given by~\eqref{eq:Short Var}.

Similar to the binary Goldbach problem, the case $H = 1$ remains 
out of reach with current methods; in fact, it is a variant of the 
well known \emph{Chowla problem}.
As a way to approximate the case $H = 1$,  it is natural to seek nontrivial bounds for $S_\mu(H,N)$
when $H$ is as small as possible relative to $N$.

The combination of Theorems~\ref{thm:main1-long} 
and~\ref{thm:noSiegel} imply the following bounds:
\begin{itemize}
\item[$(i)$] If $H\ge\er^{(\log N)^\eps}$ with some
$\eps>0$, then for any $C>0$, we have
\[
S_\mu(H,N)\llsym{\eps,C} NH(\log H)^{-C \eps^{-1}}
\le NH(\log N)^{-C}.
\]   
\item[$(ii)$] If Siegel zeros do not exist for Dirichlet
$L$-functions, then 
\[
S_\mu(H,N)\ll NH\er^{-c\sqrt{\log H}}
\]
for some effectively computable constant $c>0$.

\item[$(iii)$] If every Dirichlet $L$-function $L(s,\chi)$
is nonzero in the half-plane $\sigma\ge\sigma_\bullet$ for some
$\sigma_\bullet<1$, then
\[
\bigl|S_\mu(H,N)\bigr|
\le NH^{c(\sigma_\bullet)+o(1)} , 
\]
where $c(\sigma_\bullet)$ is as in Theorem~\ref{thm:noSiegel}. 
\end{itemize}

We note that a question of similar spirit
has also been considered for the ternary Goldbach problem,
that is, for the sums
\[
S_\Lambda(H,N)\defeq\sum_{\substack{n_1+n_2+n_3=N\\ n_3 \le H}}
\Lambda(n_1)\Lambda(n_2)\Lambda(n_3);
\]
see the work of Cai~\cite{Cai} and the references therein. However,
up to now, only much larger values of~$H$ have been handled. 
In particular, it is shown in~\cite{Cai} that
$S_\Lambda(H,N) > 0$ provided $H > N^{11/400 + \varepsilon}$ for some fixed $\varepsilon > 0$ 
and $N$ is large enough.

\section*{Acknowledgements}
We wish to thank Terry Tao and Joni Ter\"av\"ainen for helpful
conversations and for highlighting some useful references.

During the preparation of this work,  
I.S.\ was supported in part
by ARC grants DP230100530 and DP230100534.

\end{document}